\documentclass[12pt,a4paper,psamsfonts]{amsart}
\usepackage{amssymb,amscd,amsxtra,calc}
\usepackage{cmmib57}

\theoremstyle{plain}
    \newtheorem{thm}{Theorem}[section]

    \newtheorem{corollary}[thm]{Corollary}
    \newtheorem{lemma}[thm]{Lemma}
    
    \newtheorem{conjecture}[thm]{Conjecture}
    
    \newtheorem{theorem}[thm]{Theorem}

\theoremstyle{definition}

    \newtheorem{remark}[thm]{Remark}
\theoremstyle{remark}

    \newtheorem{setup}[thm]{}

\newcommand{\C}{\mathbb{C}}

\newcommand{\Q}{\mathbb{Q}}
\newcommand{\PP}{\mathbb{P}}

\newcommand{\R}{\mathbb{R}}

\newcommand{\Z}{\mathbb{Z}}

\newcommand{\OO}{\mathcal{O}}

\newcommand{\alb}{\operatorname{alb}}
\newcommand{\Alb}{\operatorname{Alb}}
\newcommand{\Amp}{\operatorname{Amp}}
\newcommand{\Aut}{\operatorname{Aut}}
\newcommand{\BIG}{\operatorname{Big}}

\newcommand{\Bir}{\operatorname{Bir}}

\newcommand{\Eff}{\operatorname{Eff}}

\newcommand{\GL}{\operatorname{GL}}

\newcommand{\id}{\operatorname{id}}
\newcommand{\Imm}{\operatorname{Im}}
\newcommand{\inter}{\operatorname{int}}
\newcommand{\Ker}{\operatorname{Ker}}
\newcommand{\Mov}{\operatorname{Mov}}
\newcommand{\cMov}{\overline{\operatorname{Mov}}}

\newcommand{\Nef}{\operatorname{Nef}}

\newcommand{\NS}{\operatorname{NS}}
\newcommand{\ord}{\operatorname{ord}}
\newcommand{\PE}{\overline{\operatorname{Eff}}}

\newcommand{\torsion}{\operatorname{torsion}}

\newcommand{\SSpec}{\operatorname{\mathit{Spec}}}

\begin{document}

\title[Birational automorphism groups of projective varieties]
{Birational automorphism groups of projective varieties of Picard number two}

\author{De-Qi Zhang}
\address
{
\textsc{Department of Mathematics} \endgraf
\textsc{National University of Singapore, 10 Lower Kent Ridge Road,
Singapore 119076
}}
\email{matzdq@nus.edu.sg}

\begin{abstract}
We slightly extend a result of Oguiso on birational automorphism groups
(resp.~of Lazi\'c - Peternell on Morrison-Kawamata cone conjecture)
from Calabi-Yau manifolds of Picard number two
to arbitrary singular varieties $X$ (resp.~to klt Calabi-Yau pairs in broad sense) of Picard number two.
When $X$ has only klt singularities and is not a complex torus, we show that
either $\Aut(X)$ is almost infinite cyclic, or it has only finitely many connected components.
\end{abstract}



\thanks{
}

\maketitle

\section{Introduction}\

This note is inspired by Oguiso \cite{Og}
and Lazi\'c - Peternell \cite{LP}.

Let $X$ be a normal projective variety defined over the field $\C$ of complex numbers.
The following subgroup (of the birational group $\Bir(X)$)
$$\Bir_2(X) := \{g : X \dasharrow X \, | \, g \,\,
\text{\rm is an isomorphism outside codimension two subsets}\}$$
is also called the group of pseudo-automorphisms of $X$.

Let $\NS(X) = \{\text{\rm Cartier divisors}\}/(\text{\rm algebraic equivalence})$
be the Neron-Severi group, which is finitely generated.
Let $\NS_{\R}(X) := \NS(X) \otimes \R$ with $\rho(X) := \dim_{\R} \NS_{\R}(X)$ the {\it Picard number}.
Let $\Eff(X) \subset \NS_{\R}(X)$ be the {\it cone of effective $\R$-divisor};
its closure $\PE(X)$ is called the {\it cone of pseudo-effective divisors}.
The {\it ample cone} $\Amp(X) \subset \NS_{\R}(X)$ consists of classes of ample $\R$-Cartier divisors;
its closure $\Nef(X)$ is called the {\it nef cone}.
A divisor $D$ is {\it movable} if $|mD|$ has no fixed component for some $m > 0$.
The {\it closed movable cone}
$\cMov(X) \subset \NS_{\R}(X)$ is the closure of the convex hull of movable divisors.
$\Mov(X)$ is the interior part of $\cMov(X)$.

A pair $(X, \Delta)$ of a normal projective variety $X$
and an effective Weil $\R$-divisor $\Delta$ is a {\it klt Calabi-Yau pair in broad sense} if
it has at worst Kawamata log terminal (klt) singularities
(cf.~\cite[Definition 2.34]{KM} or \cite[\S 3.1]{BCHM})
and $K_X + \Delta \equiv 0$
(numerically equivalent to zero);
in this case, if $K_X + \Delta$ is $\Q$-Cartier, then
$K_X + \Delta \sim_{\Q} 0$, i.e.,
$r(K_X + \Delta) \sim 0$ (linear equivalence) for some $r > 0$, by Nakayama's
abundance theorem in the case of zero numerical dimension.
$(X, \Delta)$ is a {\it klt Calabi-Yau pair in narrow sense} if
it is a klt Calabi-Yau pair in broad sense and if we assume further that
the irregularity $q(X) := h^1(X, \OO_X) = 0$.
When $\Delta = 0$, a klt Calabi-Yau pair in broad/narrow sense is called a {\it klt Calabi-Yau variety
in broad/narrow sense}.

On a terminal minimal variety (like a terminal Calabi-Yau variety) $X$,
we have $\Bir(X) = \Bir_2(X)$. Totaro \cite{To} formulated the following
generalization of the Morrison-Kawamata cone conjecture (cf.~\cite{Ka97})
and proved it in dimension two.

\begin{conjecture}\label{ConjA}
Let $(X, \Delta)$ be a $klt$ Calabi-Yau pair in broad sense.
\begin{itemize}
\item[(1)]
There exists a rational polyhedral cone $\Pi$ which is a fundamental domain for
the action of $\Aut(X)$ on the effective nef cone $\Nef(X) \cap \Eff(X)$, i.e.,
$$\Nef(X) \cap \Eff(X) = \bigcup_{g \, \in \, \Aut(X)} \, g^* \Pi ,$$
and $\inter(\Pi) \, \cap \, \inter(g^*\Pi) = \emptyset$ unless $g^*_{\,\, |\NS_{\R}(X)} = \id$.
\item[(2)]
There exists a rational polyhedral cone $\Pi'$ which is a fundamental domain for
the action of $\Bir_2(X)$ on the effective movable cone $\cMov(X) \cap \Eff(X)$.
\end{itemize}
\end{conjecture}

If $X$ has Picard number one, then $\Aut(X)/\Aut_0(X)$ is finite;
here $\Aut_0(X)$ is the {\it connected component of identity} in $\Aut(X)$;
see \cite[Prop. 2.2]{Li}.

Now suppose
that
$X$ has Picard number two.
Then $\dim_{\R} \NS_{\R}(X) = 2$.
So the (strictly convex)
cone
$\PE(X)$ has exactly two extremal rays.
Set
$$\begin{aligned}
A &:= \Aut(X), \,\,
A^- := A \setminus A^+, \,\,
B_2 := \Bir_2(X), \,\, B_2^- := B_2 \setminus B_2^+, \,\,\,\, \text{\rm where} \\
B_2^+ = \Bir_2^+(X) &:=
\{g \in B_2 \, | \, g^* \, \text{\rm preserves each of the two extremal rays of} \, \PE(X)\}, \\
A^+ = \Aut^+(X) &:=
\{g \in A \, | \, g^* \, \text{\rm preserves each of the two extremal rays of} \,  \PE(X)\} , \\
B_2^0 = \Bir_2^0(X) &:= \{g \in B_2 \, | \, g^*_{\,\, |\NS_{\R}(X)} = \id\} .
\end{aligned}$$

When $X$ is a Calabi-Yau manifold, Theorem \ref{ThB} is more or less
contained in \cite{Og} or \cite{LP}. Our argument here for general $X$ is slightly streamlined and direct.

\begin{theorem}\label{ThB}
Let $X$ be a normal projective variety of Picard number two. Then :
\begin{itemize}
\item[(1)]
$|\Aut(X) : \Aut^+(X)| \le 2$; $|\Bir_2(X) : \Bir_2^+(X)| \le 2$.
\item[(2)]
$\Bir_2^0(X)$ coincides with both $\Ker(\Bir_2(X) \to \GL(\NS_{\R}(X)))$
and
\newline
$\Ker(\Aut(X) \to \GL(\NS_{\R}(X)))$.
Hence we have inclusions:
$$\Aut_0(X) \subseteq \Bir_2^0(X) \subseteq \Aut^+(X) \subseteq \Bir_2^+(X) \subseteq \Bir_2(X).$$
\item[(3)]
$|\Bir_2^0(X) : \Aut_0(X)|$ is finite.
\item[(4)]
$\Bir_2^+(X)/\Bir_2^0(X)$ is isomorphic to either $\{\id\}$ or $\Z$.
In the former case,
$|\Aut(X) : \Aut_0(X)| \le |\Bir_2(X) : \Aut_0(X)| < \infty$.
\item[(5)]
If one of the extremal rays of $\PE(X)$ or of the movable cone of $X$ is generated by a rational divisor class
then $\Bir_2^+(X) = \Bir_2^0(X)$
and
$|\Bir_2(X) : \Aut_0(X)| < \infty$.
\item[(6)]
If one of the extremal rays of the nef cone of $X$ is generated by a rational divisor class,
then $\Aut^+(X) = \Bir_2^0(X)$ and $|\Aut(X) : \Aut_0(X)| < \infty$.
\end{itemize}
\end{theorem}

Theorem \ref{ThB} and the proof of \cite[Theorem 1.4]{LP} imply the following,
and also a weak cone theorem as in \cite[Theorem 1.4(1)]{LP} when $|\Bir_2(X) : \Aut_0(X)|$ is finite.

\begin{theorem}\label{ThC}
Let $(X, \Delta)$ be a $klt$ Calabi-Yau pair in broad sense of Picard number $2$.
Suppose that $|\Bir_2(X) : \Aut_0(X)|$ $($or equivalently $|\Bir_2^+(X) : \Bir_2^0(X)|)$
is infinite.
Then Conjecture $\ref{ConjA}$ holds true.
\end{theorem}

A group $G$ is {\it almost infinite cyclic}, if there exists an infinite cyclic subgroup $H$
such that the index $|G : H|$ is finite. If a group $G_1$ has a finite normal subgroup $N_1$
such that $G_1/N_1$ is almost infinite cyclic, then $G_1$ is
also almost infinite cyclic (cf.~\cite[Lemma 2.6]{NullG}).

\begin{theorem}\label{ThA}
Let $X$ be a normal projective variety of Picard number two. Then :
\begin{itemize}
\item[(1)]
Either $\Aut(X)/\Aut_0(X)$ is finite, or it is almost infinite cyclic and $\dim X$ is even.
\item[(2)]
Suppose that $X$ has at worst
Kawamata log terminal
singularities. Then one of the following is true.
\item[(2a)]
$X$ is a complex torus.
\item[(2b)]
$|\Aut(X) : \Aut_0(X)| \le |\Bir_2(X) : \Aut_0(X)| < \infty$.
\item[(2c)]
$X$ is a klt Calabi-Yau variety in narrow sense and
$\Aut_0(X) = (1)$, so both $\Aut(X)$ and $\Bir_2(X)$ are almost infinite cyclic.
\end{itemize}
\end{theorem}

Below is a consequence of Theorem \ref{ThA} and generalizes Oguiso \cite[Theorem 1.2(1)]{Og}.

\begin{corollary}\label{Cor1}
Let $X$ be an odd-dimensional projective variety of Picard number two.
Suppose that $\Aut_0(X) = (1)$ $($e.g., $X$ is non-ruled and $q(X) = h^1(X, \OO_X) = 0)$.
Then $\Aut(X)$ is finite.
\end{corollary}

For a linear transformation $T : V \to V$ of a vector space $V$ over $\R$ or $\C$,
the {\it spectral radius} $\rho(T)$ is defined as
$$\rho(T) := \max\{|\lambda| \, ; \, \lambda \,\,\text{\rm is a real or complex eigenvalue of} \,\, T \}.$$

\begin{corollary}
Let $X$ be a normal projective variety of Picard number two. Then:
\begin{itemize}
\item[(1)]
Every $g \in \Bir_2^+(X) \setminus \Bir_2^0(X)$
acts on $\NS_{\R}(X)$ with spectral radius $> 1$.
\item[(2)]
A class $g \Aut_0(X)$ in $\Aut(X)/\Aut_0(X)$ is of infinite order
if and only if the spectral radius of $g^*_{\,\, | \NS_{\R}(X)}$ is $> 1$.
\end{itemize}
\end{corollary}

(1) above follows from the proof of Theorem \ref{ThB},
while (2) follows from (1) and again Theorem \ref{ThB}.

\begin{remark}\label{rThA}
(1) The second alternative in Theorem \ref{ThA}(1)
and (2c) in Theorem \ref{ThA}(2) do occur.
Indeed,
the complete intersection $X$ of two general hypersurfaces of type $(1, 1)$ and $(2, 2)$
in $\PP^2 \times \PP^2$ is called Wehler's K3 surface (hence $\Aut_0(X) = (1)$)
of Picard number two such that
$\Aut(X) = \Z/(2) * \Z/(2)$ (a free product of
two copies of $\Z/(2)$) which contains $\Z$ as a subgroup of index two;
see \cite{We}.

(2) We cannot remove the possibility (2a) in Theorem \ref{ThA}(2).
It is possible that $\Aut_0(X)$ has positive dimension
and $\Aut(X)/\Aut_0(X)$ is almost infinite cyclic at the same time.
Indeed, as suggested by Oguiso,
using the Torelli theorem and the surjectivity of the period map
for abelian surfaces, one should be able to construct
an abelian surface $X$ of Picard number two with
irrational extremal rays of the
nef cone of $X$ and an automorphism $g$
with $g^*_{\,\,| \NS(X)}$ of infinite order.
Hence $g \Aut_0(X)$ is of infinite order in $\Aut(X)/\Aut_0(X)$
and $g^*$ has spectral radius $> 1$
(cf.~Corollary \ref{Cor1}).

(3) See Oguiso \cite{Og} for more examples of Calabi-Yau $3$-folds and hyperk\"ahler $4$-folds
with infinite $\Bir_2(X)$ or $\Aut(X)$.
\end{remark}

{\bf Acknowledgement.}
The author would like to thank the referee for very careful reading and valuable suggestions.
He is partly supported by an ARF of NUS.

\section{Proof of Theorems}

We use the notation and terminology in the book of Hartshorne and the book \cite{KM}.

\begin{setup}
{\bf Proof of Theorem \ref{ThB}}
\end{setup}

Since $X$ has Picard number two, we can write the pseudo-effective closed cone as
$$
\PE(X) = \R_{\ge 0}[f_1] + \R_{\ge 0}[f_2]
$$

(1) is proved in \cite{Og} and \cite{LP}.
For reader's convenience, we reproduce here.
Let $g \in B_2^-$ or $A^-$. Since $g$ permutes extremal rays of $\PE(X)$,
we can write $g^*f_1 = af_2$, $g^*f_2 = bf_1$ with $a > 0$, $b > 0$.
Since $g^*$ is defined on the integral lattice $\NS(X)/(\torsion)$, $\deg(g^*) = \pm 1$.
Hence $ab = 1$. Thus $\ord(g^*) = 2$ and $g^2 \in B_2^0$. Now (1) follows from
the observation that $B_2^- = gB_2^+$ or $A^- = gA^+$.

(2) The first equality is by the definition of $B_2^0$.
For the second equality, we just need to show that every $g \in B_2^0$ is in $\Aut(X)$.
Take an ample divisor $H$ on $X$. Then $g^*H = H$ as elements in $\NS_{\R}(X)$ over which $B_2^0$ acts trivially.
Thus $\Amp(X) \cap g(\Amp(X)) \ne \emptyset$,
where $\Amp(X)$ is the ample cone of $X$. Hence $g \in \Aut(X)$,
$g$ being isomorphic in codimension one
(cf.~e.g. \cite[Proof of Lemma 1.5]{Ka97}).

(3) Applying Lieberman \cite[Proof of Proposition 2.2]{Li} to an equivariant resolution,
$\Aut_{[H]}(X) := \{g \in \Aut(X) \, | \, g^*[H] = [H]\}$
is a finite extension of $\Aut_0(X)$ for the divisor class $[H]$ of every ample (or even nef and big)
divisor $H$ on $X$.
Since $B_2^0 \subseteq \Aut_{[H]}(X)$ (cf.~(2)), (3) follows.
See \cite[Proposition 2.4]{Og} for a related argument.

(4) For $g \in B_2^+$, write $g^*f_1 = \chi(g) f_1$ for some $\chi(g) > 0$. Then $g^*f_2 = (1/\chi(g)) f_2$
since $\deg(g^*) = \pm 1$. In fact, the spectral radius
$\rho(g^*_{\,\, |\NS_{\R}(X)}) = \max\{\chi(g), 1/\chi(g)\}$.
Consider the homomorphism
$$\varphi : B_2^+ \to (\R, +), \,\,\, g \mapsto \log \chi(g).$$
Then $\Ker(\varphi) = B_2^0$.
We claim that
$\Imm(\varphi) \subset (\R, +)$ is discrete at the origin (and hence everywhere).
Indeed, since $g^*$ acts on $\NS(X)/(\torsion) \cong \Z^{\oplus 2}$,
its only eigenvalues $\chi(g)^{\pm}$ are quadratic algebraic numbers,
the coefficients of whose minimal polynomial over $\Q$
are bounded by a function in $|\log \chi(g)|$. The claim follows.
(Alternatively, as the referee suggested,
$B_2^+ / B_2^0$ sits in
$\GL(\Z, \NS(X)/{\rm (torsion)}) \cap {\rm Diag}(f_1, f_2)$ which is a discrete group;
here ${\rm Diag}(f_1, f_2)$ is the group
of diagonal matrices with respect to the basis of $\NS_{\R}(X)$ given by $f_1$, $f_2$.)
The claim implies that
$\Imm(\varphi) \cong \Z^{\oplus r}$
for some $r \le 1$.
(4) is proved. See \cite[Theorem 3.9]{LP} for slightly different reasoning.

(5) is proved in Lemma \ref{inv} below while (6) is similar (cf.~\cite{Og}).

{\it This proves Theorem $\ref{ThB}$}.

\begin{lemma}\label{inv}
Let $X$ be a normal projective variety of Picard number two.
Then $\Bir_2^+(X) = \Bir_2^0(X)$ and hence
$|\Aut(X) : \Aut_0(X)| \le |\Bir_2(X) : \Aut_0(X)| < \infty$,
if one of the following conditions is satisfied.
\begin{itemize}
\item[(1)]
There is an $\R$-Cartier divisor $D$ such that, as elements in $\NS_{\R}(X)$, $D \ne 0$
and $g^*D = D$ for all $g \in \Bir_2^+(X)$.
\item[(2)]
The canonical divisor $K_X$ is $\Q$-Cartier, and $K_X \ne 0$ as element in $\NS_{\R}(X)$.
\item[(3)]
At least one extremal ray of $\PE(X)$, or of the movable cone of $X$
is generated by a rational divisor class.
\end{itemize}
\end{lemma}

\begin{proof}
We consider Case(1) (which implies Case(2)).
In notation of proof of Theorem \ref{ThB},
for $g \in B_2^+ \setminus B_2^0$,
we have $g^* f_1 = \chi(g) f_1$ with $\chi(g) \ne 1$ , and further $\chi(g)^{\pm 1}$ are
two eigenvalues of the action $g^*$ on $\NS_{\R}(X) \cong \R^{\oplus 2}$
corresponding to the eigevectors $f_1, f_2$.
Since $g^*D = D$ as elements in $\NS_{\R}(X)$, $g^*$ has three distinct eigenvalues: $1, \chi(g)^{\pm 1}$,
contradicting the fact: $\dim \NS_{\R}(X) = 2$.

Consider Case(3). Since every $g \in \Bir_2(X)$ acts on both of the cones, $g^2$
preserves each of the two extremal rays of both cones, one of which is rational, by
the assumption. Thus at least one of the eigenvalues of
$(g^2)^*_{\,\, | \NS_{\R}(X)}$ is a rational number
(and also an algebraic integer), so it is $1$.
Now the proof for Case(1) implies $g^2 \in B_2^0$.
So $B_2^+/B_2^0$ is trivial (otherwise, it is isomorphic to $\Z$ and torsion free by Theorem \ref{ThB}).
\end{proof}

\begin{setup}
{\bf Proof of Theorem \ref{ThA}}
\end{setup}

(1) follows from Theorem \ref{ThB}, and \cite[Proposition 3.1]{Og}
or \cite[Lemma 3.1]{LP}
for the observation that $\dim X$ is even when $\Aut^+(X)$ strictly contains $B_2^0$.

For (2), by Lemma \ref{inv}, we may assume that $K_X = 0$ as element in $\NS_{\R}(X)$.
Since $X$ is klt, $rK_X \sim 0$ for some (minimal) $r > 0$, by Nakayama's
abundance theorem in the case of zero numerical dimension.

\begin{lemma}\label{q}
Suppose that $q(X) = h^1(X, \OO_X) > 0$. Then Theorem $\ref{ThA}(2)$ is true.
\end{lemma}

\begin{proof}
Since $X$ is klt (and hence has only rational singularities)
and a complex torus contains no rational curve,
the albanese map $a = \alb_X : X \to A(X) := \Alb(X)$
is a well-defined morphism, where $\dim A(X) = q(X) > 0$; see \cite[Lemma 8.1]{Ka85}.
Further,
$\Bir(X)$ descends to
a regular action on $A(X)$, so that $a$
is $\Bir(X)$-equivariant, by the universal property of $\alb_X$.
Let $X \to Y \to A(X)$ be the Stein factorization of $a : X \to A(X)$.
Then $\Bir(X)$ descends to a regular action on $Y$.

If $X \to Y$ is not an isomorphism, then
one has that $2 = \rho(X) > \rho(Y)$, so
$\rho(Y) = 1$ and the generator of $\NS_{\R}(Y)$ gives a
$\Bir(X)$-invariant class in $\NS_{\R}(X)$. Thus
Lemma \ref{inv} applies, and Theorem \ref{ThA}(2b) occurs.

If $X \to Y$ is an isomorphism, then
the Kodaira dimensions satisfy
$\kappa(X) = \kappa(Y) \ge \kappa(a(X)) \ge 0$, by the well known fact
that every subvariety of a complex torus has non-negative Kodaira dimension.
Hence $\kappa(X) = \kappa(a(X)) = 0$, since $rK_X \sim 0$.
Thus $a$ is surjective and has connected fibres, so it is birational (cf.~\cite[Theorem 1]{Ka81}).
Hence $X \cong Y = a(X) = A(X)$, and $X$ is a complex torus.
\end{proof}

We continue the proof of Theorem \ref{ThA}(2).
By Lemma \ref{q}, we may assume that $q(X) = 0$.
This together with $rK_X \sim 0$ imply that $X$ is a klt Calabi-Yau variety in narrow sense.
$G_0 := \Aut_0(X)$ is a linear algebraic group,
by applying \cite[Theorem 3.12]{Li} to an equivariant resolution $X'$ of $X$
with $q(X') = q(X) = 0$, $X$ having only rational singularities.
The relation $rK_X \sim 0$ gives rise to the global index-one cover:
$$\hat{X} := \SSpec \, \oplus_{i=0}^{r-1} \, \OO_X(-iK_X) \to X$$
which is \'etale in codimension one, where $K_{\hat X} \sim 0$.
Every singularity of $\hat{X}$ is klt (cf.~\cite[Proposition 5.20]{KM}) and also Gorenstein, and hence canonical.
Thus $\kappa(\hat{X}) = 0$, so $\hat{X}$ is non-uniruled.
Hence $G_0 = (1)$, otherwise, since the class of $K_X$ is
$G_0$-stable, the linear algebraic group $G_0$ lifts to an action on $\hat{X}$,
so $\hat{X}$ is ruled, a contradiction. Thus Theorem \ref{ThA} (2b) or (2c) occurs (cf. Theorem \ref{ThB}).

{\it This proves Theorem $\ref{ThA}$}.

\begin{setup}
{\bf Proof of Theorem \ref{ThC}.}
It follows from the arguments in \cite[Theorem 1.4]{LP}, Theorem \ref{ThB}
and the following (replacing \cite[Theorem 2.5]{LP}):
\end{setup}

\begin{lemma}
Let $(X, \Delta)$ be a $klt$ Calabi-Yau variety in broad sense. Then
both the cones $\Nef(X)$ and $\cMov(X)$ are locally rational polyhedral inside the cone
$\BIG(X)$ of big divisors.
\end{lemma}

\begin{proof}
Let $D \in \cMov(X) \cap \BIG(X)$. Since $(X, \Delta)$ is klt
and klt is an open condition,
replacing $D$ by a small multiple,
we may assume that $(X, \Delta + D)$ is klt. By \cite[Theorem 1.2]{BCHM},
there is a composition $\sigma: X \dasharrow X_1$ of divisorial and flip birational contractions
such that $(X_1, \Delta_1 + D_1)$ is klt and $K_{X_1} + \Delta_1 + D_1$ is nef; here
$\Delta_1 := \sigma_*\Delta_1$, $D_1 := \sigma_*D$, and
$K_{X_1} + \Delta_1 = \sigma_*(K_X + \Delta) \equiv 0$.
Since $K_X + \Delta + D \equiv D \in \cMov(X)$, $\sigma$ consists only of flips, so $D = \sigma^*D_1$.
By \cite[Theorem 3.9.1]{BCHM}, $(K_{X_1} + \Delta_1) + D_1$ ($\equiv D_1$)
is semi-ample (and big),
so it equals $\tau^*D_2$, where $\tau : X_1 \to X_2$ is a birational morphism
and $D_2$ is an $\R$-Cartier ample divisor.
Write $D_2 = \sum_{i=1}^s c_i H_i$ with $c_i > 0$ and $H_i$ ample and Cartier.
Then $D \equiv \sum_{i=1}^s c_i \sigma^*\tau^*H_i$ with $\sigma^*\tau^*H_i$ movable and Cartier.
We are done (letting $\sigma = \id$ when $D \in \Nef(X) \cap \BIG(X)$).
Alternatively, as the referee suggested,
in the case when $D_2$ lies on the boundary
of the movable cone, fix a rational effective
divisor $E$ close to $D_2$ outside the movable
cone - but still inside the big cone. Then,
for $\varepsilon \in \Q_{\ge 0}$ small enough,
$\varepsilon E \equiv K_X + \Delta + \varepsilon E$
is klt and a rational divisor. Taking $H$ an ample divisor,
the rationality theorem in \cite[Theorem 3.5, and Complement 3.6]{KM} shows that
the ray spanned by $D_2$ is rational.
\end{proof}

\end{document}